\documentclass[a4paper,12pt]{article}
\usepackage[centertags]{amsmath}
\usepackage{amsfonts}
\usepackage{amssymb}
\usepackage{amsthm}
\usepackage{amsmath}
\usepackage{dsfont}
\usepackage{graphicx} 
\usepackage{tikz, subfigure}
\usepackage{pstricks-add}
\usepackage{verbatim}

\usepackage[latin1]{inputenc}
\usepackage{tikz}
\definecolor {processblue}{cmyk}{0.96,0,0,0}
\usepackage{amsfonts,graphicx,amsmath,amssymb,hyperref,color}
\usepackage[english]{babel}
\usepackage{authblk}

\newtheorem{Theorem}{Theorem}[section]
\newtheorem{Proposition}{Proposition}[section]
\newtheorem{Lemma}{Lemma}[section]
\newtheorem{Corollary}{Corollary}[section]

\newtheorem{Definition}{Definition}[section]

\newtheorem{Example}{Example}[section]

\newtheorem{Remark}{Remark}[section]

\newtheorem{Question}{Question}[section]

\AtEndDocument{\bigskip{\footnotesize%
  \textsc{Department of Mathematics, Ningbo University, Ningbo, Zhejiang, People's Republic of China} \par
  \textit{E-mail address:} \texttt{kanjiangbunnik@yahoo.com, jiangkan@nbu.edu.cn} \par
}}

\usepackage{lipsum}% http://ctan.org/pkg/lipsum

\makeatletter
\newcommand*{\rom}[1]{\expandafter\@slowromancap\romannumeral #1@}
\makeatother

\begin{document}
\title{xxxx}
\date{}
 \title{Projections of cartesian products of the self-similar sets without the irrationality assumption}
\author{Kan Jiang }
\maketitle{}
\begin{abstract}
Let $\beta>1$. Define a class of similitudes
\[S=\left\{f_{i}(x)=\dfrac{x}{\beta^{n_i}}+a_i:n_i\in \mathbb{N}^{+}, a_i\in \mathbb{R}\right\}.\]
Let $\mathcal{A}$ be the collection of all the self-similar sets generated by the  similitudes from $S$.
In this paper, we  prove that for any $\theta\in[0,\pi)$ and  $K_1, K_2\in \mathcal{A}$,
$Proj_{\theta}(K_1\times K_2)$ is similar to a self-similar set or an attractor of some infinite iterated function system, where  $Proj_{\theta}$ denotes the orthogonal projection onto  $L_{\theta}$, and  $L_{\theta}$  denotes  the line through the origin in direction $\theta$.
As a corollary, $\dim_{P}(Proj_{\theta}(K_1\times K_2))=\overline{\dim}_{B}(Proj_{\theta}(K_1\times K_2))$ holds for  any $\theta\in[0,\pi)$ and any $K_1, K_2\in \mathcal{A}$,  where $\dim_{P}$ and $\overline{\dim}_{B}$ denote the packing and upper box dimension. 
 Whether $Proj_{\theta}(K_1\times K_2)$ is similar to a self-similar set or not is uniquely determined by the similarity ratios of $K_1$ and $K_2$ rather than the angle $\theta.$
  When $Proj_{\theta}(K_1\times K_2)$ is similar to a self-similar set, in terms of the finite type condition \cite{NW}, we are able to calculate in cerntain cases  the Hausdorff dimension of 
$Proj_{\theta}(K_1\times K_2)$. If $Proj_{\theta}(K_1\times K_2)$ is similar to an attractor of some infinite iterated function system, then by virtue of the Vitali covering lemma \cite{FG} we give an estimation of the Hausdorff dimension of 
$Proj_{\theta}(K_1\times K_2)$. For some cases, we can calculate, by means  of Mauldin and Urbanski' result \cite{MRD},  the exact Hausdorff dimension of $Proj_{\theta}(K_1\times K_2)$. 
 We also find some  non-trivial  examples such that for some angle $\theta\in[0,\pi)$ and some $K_1, K_2\in \mathcal{A}$, $\dim_{H}(Proj_{\theta}(K_1\times K_2))=\dim_{H}(K_1)+\dim_{H}(K_2)$. 
\end{abstract}
\section{Introduction}
Let $L_{\theta}$  be   the line through the origin in direction $\theta$, and 
 $Proj_{\theta}$ denotes the orthogonal projection onto  $L_{\theta}$. Given two Borel sets $A,B\in \mathbb{R}$,
analyzing  the set $Proj_{\theta}(A\times B)$ is a crucial topic in geometric measure theory.  The classical Marstrand theorem \cite{FG} states that 
\begin{Theorem}
Given two Borel sets $A,B\in \mathbb{R}$.
\begin{itemize}
\item [(1)] If $\dim_{H}(A)+\dim_{H}(B)\leq 1$, then for almost all $\theta\in[0, \pi)$,  $$\dim_{H}(Proj_{\theta}(A\times B))=\dim_{H}(A)+\dim_{H}(B);$$ 
\item [(2)] If $\dim_{H}(A)+\dim_{H}(B)>1$, then for almost all $\theta\in[0, \pi)$,  $Proj_{\theta}(A\times B)$ has positive Lebesgue measure. 
\end{itemize}
\end{Theorem}
Unfortunately,  Marstrand theorem does not offer any information for a specific angle $\theta.$  
For the self-similar sets, 
Peres and Shmerkin \cite{PS}, Hochman and Shmerkin \cite{Hochman2012} proved the following elegent result.
\begin{Theorem}\label{Thm2}
Let $K_1$ and $K_2$ be two self-similar sets  with IFS's $\{f_i(x)=r_ix+a_i\}_{i=1}^{n}$ and  $\{g_j(x)=r_j^{\prime}x+b_j\}_{j=1}^{m}$, respectively. If for any $r_i, r_j^{\prime}$, $$\dfrac{\log |r_i| }{\log |r_j^{\prime}|}\notin \mathbb{Q},$$ then 
$$\dim_{H}(K_1+K_2)=\min\{\dim_{H}(K_1)+\dim_{H}(K_2),1\},$$
and $$\dim_{H}(K_1+K_2)=\dim_{P}(K_1+K_2)=\dim_{B}(K_1+K_2).$$
\end{Theorem}
The condition in Theorem \ref{Thm2} is called the irrationality assumption. Note that $K_1+K_2$ is similar  to $Proj_{\pi/4}(K_1\times K_2)$. Therefore, Theorem \ref{Thm2}  states that under the irrationality assumption, the Hausdorff dimension of the projection of two self-similar sets through the angle $\pi/4$ does not decrease. Peres and Shmerkin  indeed \cite{PS} proved a  general result in $\mathbb{R}^2$, i.e. if  the group generated by the rotations of IFS is dense in $[0,\pi)$, then for any angle $\theta\in[0,\pi)$,  the  Hausdorff dimension  of the projection of the attractor coincides with the expected  Hausdorff dimension.
However, without the irrationality assumption, generally the dimension  of $Proj_{\theta}(K_1\times K_2)$ may drop.  In this paper, we consider  the following class of similitudes:
let $\beta>1$,  define a class of similitudes
\[S:=\left\{f_{i}(x)=\dfrac{x}{\beta^{n_i}}+a_i:n_i\in \mathbb{N}^{+}, a_i\in \mathbb{R}\right\}.\]
Let $\mathcal{A}$ be the collection of all the self-similar sets generated by the  similitudes from $S$.  In \cite{PS}, Peres and Shmerkin proved
the following result.
\begin{Theorem}
For any $K_1, K_2\in \mathcal{A}$ such that their Hausdorff dimensions coincide with the associated similarity dimensions, then  there exists some $\theta\in[0,\pi)$ such that 
$$\dim_{H}(Proj_{\theta}(K_1\times K_2))<\min\{1,\dim_{H}(K_1)+\dim_{H}(K_2)\}.$$
\end{Theorem}
Generally, the Hausdorff dimension of   $Proj_{\theta}(K_1\times K_2)$ is difficult to  calculate. The main aim of this paper is to analyze the set $Proj_{\theta}(K_1\times K_2)$, and give an estimation of its Hausdorff dimension. 

The following  are the main results of this paper. 
\begin{Theorem}\label{Main}
Given any $\theta\in[0,\pi)$,  and any $K_1, K_2\in \mathcal{A}$, $Proj_{\theta}(K_1\times K_2)$ is similar to a self-similar set or an attractor of some infinite iterated function system. 
\end{Theorem}
In terms of Theorem \ref{Main}, we have the following corollaries.
\setcounter{Corollary}{4}
\begin{Corollary}
 For any $\theta\in[0,\pi)$ and any $K_1, K_2\in \mathcal{A}$, $$\dim_{P}(Proj_{\theta}(K_1\times K_2))=\overline{\dim}_{B}(Proj_{\theta}(K_1\times K_2)).$$
\end{Corollary}
\begin{Corollary}\label{Cor}
Given $k\geq 1$. Suppose that $\beta$ is a Pisot number. Let $K_1$ be the attractor of the following IFS 
$$\left\{f_i(x)=\dfrac{x}{\beta^{k}}+a_i,1\leq i\leq n\right\},$$
and $K_2$ be the attractor of the following IFS 
$$\left\{g_j(x)=\dfrac{x}{\beta^{l_jk}}+b_j,1\leq j\leq m\right\},$$
where $l_j\in\mathbb{N}^{+}$. If $a_i, b_j,\tan\theta\in \mathbb{Z}[\beta], 1\leq i\leq n,1\leq j\leq m$, then $Proj_{\theta}(K_1\times K_2)$ is similar to a self-similar set with  the finite type condition \cite{NW}. Moreover, the Hausdorff dimension of $Proj_{\theta}(K_1\times K_2)$ can be calculated explicitly. 
\end{Corollary}
\begin{Corollary}\label{IIFSS}
Given $\theta\in[0,\pi)$, and $K_1, K_2\in \mathcal{A}$.
Suppose that 
$Proj_{\theta}(K_1\times K_2)$ is similar to an attractor with infinite iterated function system. Then 
 there exist two attractors $J_1, J_2$ with infinite iterated function systems
such that 
$$s_1(\theta)\leq \dim_{H}(Proj_{\theta}(K_1\times K_2)) \leq s_2(\theta),$$
where $s_1(\theta)$ is the Hausdorff dimension of $J_1$, and $s_2(\theta)$ is the similarity dimension of $J_2.$
\end{Corollary}
For some cases,  even though  $Proj_{\theta}(K_1\times K_2)$ is similar to some attractor with infinite iterated function system which does  not satisfy the open set condition, we can still calculate the  exact Hausdorff dimension of $Proj_{\theta}(K_1\times K_2)$. The following example is  one of the  typical cases. 
\setcounter{Example}{7}
\begin{Example}\label{Example}
Let $K_1=K_2$ be the attractor of the IFS $$\left\{f_1(x)=\dfrac{x}{\beta^4},f_2(x)=\dfrac{x+\beta^8-1}{\beta^8}\right\}.$$
Suppose that $\beta>1.39$, then  for
any $\theta\in\left(\arctan \dfrac{\beta^{12}-\beta^{8}+1}{\beta^{12}-\beta^8-\beta^4}, \arctan \dfrac{\beta^{12}-2\beta^4}{\beta^{12}-\beta^8+1}\right)$ 
$$\dim_{H}(Proj_{\theta}(K_1\times K_2))=\dfrac{\log \sqrt{\dfrac{1+\sqrt{5}}{2}}}{\log \beta}=\dim_{H}(K_1)+\dim_{H}(K_2).$$
Let $\theta=\arctan\dfrac{\beta^{8}-1}{\beta^8-\beta^4+1}$ and $\beta>1.41$. 
Then $$\dim_{H}(Proj_{\theta}(K_1\times K_2))=\dfrac{\log \gamma}{\log \beta}<\dim_{H}(K_1)+\dim_{H}(K_2),$$ where $\gamma\approx 1.2684$ is the largest  real root of 
$$x^{20}-2x^{16}-2x^{12}+x^{8}+x^{4}-1=0.$$
\end{Example}
We can  find similar examples as Example \ref{Example} and calculate the Hausdorff dimension of 
$Proj_{\theta}(K_1\times K_2)$ for some explicit angle $\theta$. 

 This paper is arranged as follows.  In section 2, we give the proofs of the main results. In section 3 we analyze  Example \ref{Example}. Finally, we give some remarks.
\section{Proof of Theorem \ref{Main}}
\subsection{Preliminaries  and some key lemmas}
In this section, we shall prove that $Proj_{\theta}(K_1\times K_2)$ is similar to a self-similar set or  an attractor with infinite iterated function system. First, we introduce some definitions and results. 
The definition of self-similar set is due to Hutchinson \cite{Hutchinson}. Let $K$ be the self-similar  set of the IFS $\{f_i\}_{i=1}^{m}$.
For any  $x \in K$, there exists a sequence
$(i_n)_{n=1}^{\infty}\in\{1,\ldots,m\}^{\mathbb{N}}$ such that
\[x=\lim_{n\to \infty}f_{i_1}\circ \cdots\circ f_{i_n}(0).\] We call $(i_n)$  a coding of $x$.  We can define a surjective projection map between the symbolic space $\{1,\ldots, m\}^{\mathbb{N}}$ and the self-similar set $K$ by
 $$\pi((i_n)_{n=1}^{\infty}):=\lim_{n\to \infty}f_{i_1} \circ\cdots \circ
f_{i_n}(0).$$
Usually, the coding of $x$ is not unique \cite{DJKL,KarmaKan2}.
Given two self-similar sets $K_1$ and $K_2$ from the class $\mathcal{A}$.
Suppose that  the  IFS's of $K_1$ and $K_2$ are
$\{f_{i}(x)=\frac{x}{\beta^{n_i}}+a_i\}_{i=1}^{n}$ and
  $\{g_{j}(x)=\frac{x}{\beta^{m_j}}+b_j\}_{j=1}^{m}$, respectively. 
 Note
that
\[f_{i}(x)=\frac{x}{\beta^{n_i}}+a_i=\frac{x+\beta^{n_i}a_i}{\beta^{n_i}}=\frac{x}{\beta^{n_i}}+\frac{0}{\beta}+\frac{0}{\beta^2}+\cdots+
\frac{0}{\beta^{n_i-1}}+\frac{\beta^{n_i}a_i}{\beta^{n_i}}.\] Therefore, we can  identify  $f_i(x)$ with a block $(\underbrace{000\cdots 0}_{n_i-1}a_i^{'}
)$,\, where $a_i^{'}=\beta^{n_i}a_i$. Conversely, any block $(\underbrace{000\cdots 0}_{n_i-1}a_i^{'}
)$ can determine a unique similitude with respect to  $\beta$. 
  For simplicity  we denote this block by
$\hat{P}_{i}=(\underbrace{000\cdots 0}_{n_i-1}a_i^{'}
)$.  In what follows, we  identify  $f_{i}$ with $f_{\hat{P}_{i}}$. Similarly, we may define blocks in terms of the IFS of $K_2$. Let $D_1=\{\hat{P}_{1},\, \hat{P}_{2},\,\cdots,\,\hat{P}_{n}\}$ and $D_2=\{\hat{Q}_{1},\, \hat{Q}_{2},\,\cdots,\,\hat{Q}_{m}\}$, where $\hat{P}_{i}=(\underbrace{000\cdots 0}_{n_i-1}a_i^{'}
)$, $a_i^{'}=\beta^{n_i}a_i$,  $\hat{Q}_{j}=(\underbrace{000\cdots 0}_{m_j-1}b_j^{'}
)$ and $b_j^{'}=\beta^{m_j}b_j$. The following lemma is trivial. 
\begin{Lemma}\label{coding}
\[K_{1}=\{x=\lim\limits_{n\to \infty}f_{\hat{P}_{i_1}}\circ f_{\hat{P}_{i_2}} \circ\cdots\circ f_{\hat{P}_{i_n}}(0):\hat{P}_{i_j}\in D_1\}.\]
\[K_{2}=\{y=\lim\limits_{n\to \infty}g_{\hat{Q}_{i_1}}\circ g_{\hat{Q}_{i_2}} \circ\cdots\circ g_{\hat{Q}_{i_n}}(0):\hat{Q}_{i_j}\in D_2\}.\]
 We call  the  infinite concatenation $\hat{P}_{i_1}\ast \hat{P}_{i_2} \ast \cdots$ ($\hat{Q}_{i_1}\ast \hat{Q}_{i_2} \ast \cdots$) a coding of x ($y$). 
\end{Lemma}
\begin{Lemma}\label{conjugate}
For any $\theta\in(0,\pi)\setminus{\pi/2}$, 
$Proj_{\theta}(K_1\times K_2))$ is similar to $K_1+ \tan(\theta)K_2.$  $Proj_{\pi/2}(K_1\times K_2))=K_2.$
\end{Lemma}
\begin{proof}
Note that  $Proj_{\theta}(x,y)$ is  point on the line $L_{\theta}$ at distance  $$x\cos\theta+y\sin\theta=(x+sy)\cos\theta,$$
where $s=\tan\theta.$
\end{proof}
By Lemma   \ref{conjugate}, if we want to analyze $(Proj_{\theta}(K_1\times K_2))$, it suffices to consider the set $K_1+sK_2,$ where $s=\tan\theta.$

The infinite iterated function systems (IIFS)  play a pivotal role in this paper,  we first introduce some  definitions  and related  results  of this powerful tool. 

There are two definitions of the invariant set of  IIFS, see for example, \cite{HF},\, \cite{MRD} and \cite{HM}. We adopt Fernau's definition \cite{HF}.
\setcounter{Definition}{2}
\begin{Definition}\label{IIIFS}
Let $\mathcal{A}=\{\phi_{i}(x)=r_ix+a_i: i\in \mathbb{N}^{+},\,0<r_i<1, a_i\in \mathbb{R}\}$. Suppose that  there exists a uniform $0<c<1$ such that for
every $\phi_{i}\in \mathcal{A}$, $|\phi_{i}(x)-\phi_{i}(y)|\leq c
|x-y| $, then  we say $\mathcal{A}$ is an infinite iterated function system. The  unique non-empty compact  set  $J$  is called  the attractor (or invariant set) of $\mathcal{A}$ if
\[J=\overline{\bigcup_{i\in \mathbb{N}}\phi_{i}(J)},\]
where $ \overline{A}$ denotes the closure of $A$. We call  $s_0$,  which is the unique solution of the equation $\sum_{i=1}^{\infty}r_i^s=1$,  the similarity dimension of $J$. 
\end{Definition}
\setcounter{Remark}{1}
 In \cite{MRD}, Mauldin and Urbanski gave another definition of the attractor of IIFS, i.e. 
 $$J_0\triangleq
\bigcup\limits_{\{\phi_{i_n}\}\in\mathcal{A}}\bigcap\limits_{n=1}^{\infty}\phi_{i_1}\circ\phi_{i_2}\cdots\circ
\phi_{i_n}([0,1]),$$
which yields that 
 $J_{0}=\bigcup_{i\in \mathbb{N}}\phi_{i}(J_{0})$. However, for this definition the attractor $J_{0}$ may not be unique  or compact, see  example 1.3 from \cite{HF}. Evidently, $\overline{J_0}=J$. In what follows,   $J_0$ means that the attractor is in the sense of Mauldin and Urbanski while  $J$ refers to Fernau's definition.

 The following result can be found in \cite{MRD,M,HM}. We shall utilize this result to calculate the Hausdorff dimension of $Proj_{\theta}(K_1\times K_2).$
\setcounter{Theorem}{3}
\begin{Theorem}\label{DimensionIIFS}
Let $J_0$ be the attractor of some IIFS with the open set condition, then
\[
\dim_{H}(J_0)=\inf\left\{t:\sum_{i\in \mathbb{N}}r_i^{t}\leq 1\right\}.
\]
\end{Theorem}

The following definitions are defined in a natural way. 
\setcounter{Definition}{4}
\begin{Definition}
Let $\Sigma=\{s_1,\cdots, s_p \}$, where  $s_i\in \mathbb{R}, 1\leq i\leq p$. 
 Let $d_1d_2\cdots d_k$ and $c_1c_2\cdots c_k$ be two blocks from $\{s_1,\cdots,s_p\}^{k}$.  We say  the block  $d_1d_2\cdots d_k$ is of length $k$. Define the concatenation of $d_1d_2\cdots d_k$ and $c_1c_2\cdots c_k$ by   $$(d_1d_2\cdots d_k)* (c_1c_2\cdots c_k)=d_1d_2\cdots d_kc_1c_2\cdots c_k.$$ The sum of $d_1d_2\cdots d_k$ and $c_1c_2\cdots c_k$ is defined by $(d_1+c_1)(d_2+c_2)\cdots (d_k+c_k)$.  The concatenation of $t\in \mathbb{N}$ blocks  with   $\hat{P}_{1}=d_1d_2\cdots d_k $
  is denoted by  \[\hat{P}_{1}^{t}=\underbrace{\hat{P}_{1}*\hat{P}_{1}*\cdots*\hat{P}_{1}}_\text{$t$ times}.\] The value of the block
$\hat{P}_{1}=d_1d_2\cdots d_k$ in base $\beta>1$ is defined as \[(d_1d_2\cdots d_k)_{\beta}= \dfrac{d_1}{\beta}+\dfrac{d_2}{\beta^2}+\cdots+\dfrac{d_k}{\beta^k}.\]
Similarly, given $(d_n)\in \{s_1\cdots, s_p \}^{\mathbb{N}}$, define  $$(d_n)_{\beta}=\sum\limits_{n=1}^{\infty}\dfrac{d_n}{\beta^n}.$$
\end{Definition}
  Recall that  $D_2=\{\hat{Q}_{1},\, \hat{Q}_{2},\,\cdots,\,\hat{Q}_{m}\}$, where   $\hat{Q}_{j}=(\underbrace{000\cdots 0}_{m_j-1}b_j^{'}
)$ and $b_j^{'}=\beta^{m_j}b_j$. 
We define a new  set $D_2^{\prime}=\{s\hat{Q}_{1},\, s\hat{Q}_{2},\,\cdots,\,s\hat{Q}_{m}\}$, where $s\hat{Q}_{j}=(\underbrace{000\cdots 0}_{m_j-1}sb_j^{'}
)$ and $b_j^{'}=\beta^{m_j}b_j$, $s=\tan\theta$ is from Lemma \ref{conjugate}. 
The following definition was  essentially given in \cite{SumKan}, we slightly modify the definition. 
\begin{Definition}\label{Matching}
Take $u$ blocks
\[\hat{P}_{i_1},\,\hat{P}_{i_2},\,\hat{P}_{i_3},\,\cdots,\,
\hat{P}_{i_{u}}\] from $D_1$ with  lengths
$p_{1},\,p_{2},\,p_{3},\,\cdots,\, p_{u}$, respectively. Pick $v$
blocks \[s\hat{Q}_{j_1},\,s\hat{Q}_{j_2},\,s\hat{Q}_{j_3},\,\cdots,\,
s\hat{Q}_{j_{v}}\] from $D_2^{\prime}$ with lengths
$q_{1},\,q_{2},\,q_{3},\,\cdots ,\,q_{v}$, respectively.  If there exist
integers
 $k_{1},\,k_{2},\,k_{3},\cdots,
k_{u}$,\\ $l_{1},\,l_{2},\,l_{3},\cdots, l_{v}$ such that
\[
\sum_{i=1}^{u}k_{i}p_{i}=\sum_{j=1}^{v}l_{j}q_{j},
\]
then we call $A+B$ 
 a Matching   with respect to $\beta$, where 
 $$A =\hat{P}_{i_{t_1}}\ast \hat{P}_{i_{t_2}}\ast\cdots\ast
\hat{P}_{i_{t_u}},$$
and  there are precisely $k_p$  $\hat{P}_{i_{p}}$'s in the concatenation $\hat{P}_{i_{t_1}}\ast \hat{P}_{i_{t_2}}\ast\cdots\ast
\hat{P}_{i_{t_u}}$,
 $$B=s\hat{Q}_{j_{w_1}}\ast s\hat{Q}_{j_{w_2}}\ast\cdots\ast
s\hat{Q}_{j_{w_v}},$$ 
and  there are precisely $l_q$  $(s\hat{Q}_{j_{q}})$'s in the  concatenation $s\hat{Q}_{j_{w_1}}\ast s\hat{Q}_{j_{w_2}}\ast\cdots\ast
s\hat{Q}_{j_{w_v}},$
 where $1\leq p\leq u$, $t_i\in\{1,2\cdots,u\}, 1\leq i\leq u$, $1\leq q\leq v$, $w_j\in\{1,2\cdots,v\}, 1\leq j\leq v$.
\end{Definition}
\setcounter{Remark}{6}
\begin{Remark}
In \cite{SumKan}, the definition of Matching is incorrect. We need a little modification.  Due to the condition $\sum_{i=1}^{u}k_{i}p_{i}=\sum_{j=1}^{v}l_{j}q_{j},$ it follows that the lengths  of $A$ and   $B$  coincide. Therefore, we can define the sum of these two concatenated blocks.  A Matching is also a  block which is the sum of some concatenated blocks from $D_1$ and $D_2^{\prime}$, respectively. To avoid some unnecessary Matchings,  in what follows, we always  obey the following rule,   i.e.  if the new born
Matchings can be concatenated by the old Matchings,
then we do not choose these new Matchings.
\end{Remark}
\setcounter{Example}{7}
\begin{Example}
 Given  $\beta>1$.  Let $K_1=K_2$ be the attractor of the IFS $$\left\{f_1(x)=\dfrac{x}{\beta^4},f_2(x)=\dfrac{x+\beta^8-1}{\beta^8}\right\}.$$
Denote $A=\beta^8-1,B=sA, C=A+B, s=\tan\theta$.
$$D_1=\{(0000), (0000000A)\}, D_2^{\prime}=\{(0000), (0000000B)\}.$$ 
All the Matchings generated by $D_1$ and $D_2^{\prime}$ is 
 $$D=\{(0000), (0000000A), (0000000B),(0000000C), (0000000B000A),(0000000A000B),\cdots\}.$$
 Note that in this example the lengths of the Matchings should be $4k, k\in \mathbb{N}^{+}$ due to the lengths of blocks from $D_1$ and $D_2^{\prime}$.  Clearly, the block $$(00000000000A)=(0000)\ast(0000000A),$$ i.e. the block $(00000000000A)$ can be concatenated by another two Matchings. For such case, we do not take $(00000000000A)$ as a Matching. 
\end{Example}
The following result can be found in \cite{SumKan}.  
\setcounter{Lemma}{8}
\begin{Lemma}\label{GenerateMatchings}
The cardinality of Matchings which are generated by $D_1$ and $D_2^{\prime}$ is at most countable.
\end{Lemma}
Denote all the Matchings by
\[D=\{\hat{R}_{1} ,\,\hat{R}_{2},\,\cdots,\,\hat{R}_{n-1},\,
\hat{R}_{n},\cdots\}. \]  Since each Matching determines a similitude with respect to $\beta$ (the approach is the same as we identify each similitude of $K_1$ with some block $\hat{P}_i$), it follows that  $D$ uniquely determines a set of similitudes
$\Phi^{\infty}\triangleq\{\phi_1,\,\phi_2,\,\phi_3,\,\phi_4,\cdots\}
$. If the cardinality of $D$ is finite, then $K_1+sK_2$ is clearly a self-similar set. We will prove this fact in the next subsection.  If $\sharp D$ is infinite, then we define 
$$E\triangleq
\bigcup\limits_{\{\phi_{i_n}\}\in\Phi^{\infty}}\bigcap\limits_{n=1}^{\infty}\phi_{i_1}\circ\phi_{i_2}\cdots\circ
\phi_{i_n}([0,1]),$$ and  $E$ is a solution of the equation
$E=\bigcup\limits_{i\in \mathbb{N}}\phi_{i}(E)$, \cite{MRD}.  
\subsection{Proof of Theorem \ref{Main}}
First we  assume that the cardinality of all Matchings is infinitely countable.
  In Lemma \ref{coding} we give a new  definition of the codings of $K_i$, $1\leq i\leq 2$.  For any  $x+sy\in K_1+sK_2$, 
 we denote the coding of $x+sy$ by $(x_n+sy_n)_{n=1}^{\infty}$, where $(x_n)$ and $(y_n)$ are the codings of $x$ and  $y$, respectively.  By Lemma \ref{coding}, 
We know that $(x_n)$ ($(sy_n)$) can be decomposed into infinite blocks from $D_1$($D_2^{\prime}$), 
namely, $(x_n)=X_1\ast X_2\ast \cdots$ and $(sy_n)=sY_1\ast sY_2\ast \cdots.$

Let $(a_n)_{n=1}^{\infty}=(x_n+sy_n)$ be a coding of some point $x+sy\in K_1+sK_2$,  where $(x_n)^{\infty}_{n=1}$ and $(y_n)^{\infty}_{n=1}$ are the codings of $x\in K_1$ and $y\in K_2$, respectively. Given $k>0$,  we call $(c_{i_1}c_{i_2}\cdots c_{i_k})$  a word of $(a_{i})_{i=1}^{\infty}$  with length $k$ if  there exists  some $j>0$ such that  $c_{i_1}c_{i_2}\cdots c_{i_k}=a_{j+1}\cdots a_{j+k}$. Let 
 \begin{align*}
C=\Big\{(a_n)=(x_n+sy_n): &\textrm{ there exists some } \,N\in \mathbb{N}^{+}\,\textrm{such that any word of}\\
& (a_{N+i})_{i=1}^{\infty} \textrm{ is not a Matching}\Big\}.
\end{align*}
\setcounter{Lemma}{9}
\begin{Lemma}\label{app}
Let $(a_n)\in C$, for any $\epsilon>0$ we can find a coding $(b_n)_{n=1}^{\infty}$ which is the concatenation of infinite Matchings such that
\[|(a_n)_{\beta}-(b_n)_{\beta} |< \epsilon.\]
\end{Lemma}

\begin{proof}
Let $(a_n)\in C$. For any  $\epsilon>0$,  there exists some  $n_0\in \mathbb{N}$ such that  $\beta^{-n_0}< \epsilon$. We will define some   $(b_n)_{n=1}^{\infty}$ such that its value in base $\beta$ is a point of $E$.  

\textbf{Case 1.} 
Suppose  $a_1a_2a_3\cdots a_{n_0}$ is a Matching or a concatenation of some Matchings, then we can choose any $(b_{n_0+i})_{i=1}^{\infty}$ that is the concatenation of infinite  Matchings. Therefore, 
\[
|(a_n)_{\beta}-(b_n)_{\beta} |= |(a_{n_0+1}a_{n_0+2}a_{n_0+3}\cdots)_{\beta}-(b_{n_0+1}b_{n_0+2}b_{n_0+3}\cdots)_{\beta}|\leq M \sum_{i=n_0+1}^{\infty}\beta^{-i}=M^{\prime}\epsilon,
\]
where $M, M^{\prime}$ are  positive constants. 
Therefore,  we have  proved that there exists some point $ b\in E$ such that
\[|(a_n)_{\beta}-(b_n)_{\beta} |< \epsilon.\]
\textbf{Case 2.} 
If $a_1a_2a_3\cdots a_{n_0}$ is not a concatenation of some Matchings,  by  virtue of the definition of $(a_n)$, 
$(a_n)=(x_n+sy_n)$, where $(x_n)=(X_1\ast X_2\ast \cdots), (sy_n)=(sY_1\ast sY_2\ast \cdots) $ are the codings of some  points in $K_1$ and $K_2$, respectively.  Suppose that there exist $p, q$ such that $a_1a_2a_3\cdots a_{n_0}$
 is a prefix of $(X_1\ast X_2\ast \cdots \ast X_p)+(sY_1\ast sY_2\ast \cdots \ast sY_q)$, the lengths of $X_1\ast X_2\ast \cdots \ast X_p$ and $sY_1\ast sY_2\ast \cdots \ast sY_q$ may not coincide. Nevertheless, we may still  define the summation of their common prefixes. Assume that the length of $X_1\ast X_2\ast \cdots \ast X_p$   and   $sY_1\ast sY_2\ast \cdots \ast sY_q$ are  $k_1$ and $k_2$, respectively. Then $$(X_1\ast X_2\ast \cdots \ast X_p)^{k_2}+(sY_1\ast sY_2\ast  \cdots \ast sY_q)^{k_1}$$ is a Matching or a concatenation of some Matchings as the blocks $(X_1\ast X_2\ast \cdots \ast X_p)^{k_2}$ and $(sY_1\ast sY_2\ast  \cdots \ast sY_q)^{k_1}$ have the same length.  Moreover, the initial $n_0$ digits of  $(X_1\ast X_2\ast \cdots \ast X_p)^{k_2}+(sY_1\ast sY_2\ast  \cdots \ast sY_q)^{k_1}$ is $a_1a_2a_3\cdots a_{n_0}$.  Now, we can  make use of the idea in the first case. 
\end{proof}
\begin{Lemma}\label{Closure}
$\overline{E}=K_1+sK_2$.
\end{Lemma}
\begin{proof} For any $ \epsilon>0$ and  any $
x+sy\in K_1+sK_2$,  we can find a coding $(a_n)$ such that  $x+sy=\sum\limits_{n=1}^{\infty}a_n\beta^{-n}$. If there exists a subsequence of integer $n_{k}\to \infty$ such that $(a_1a_2a_3 \cdots a_{n_k})$ is  always a concatenation of some Matchings,  then by the definition of $$E=
\bigcup\limits_{\{\phi_{i_n}\}\in\Phi^{\infty}}\bigcap\limits_{n=1}^{\infty}\phi_{i_1}\circ\phi_{i_2}\cdots
\phi_{i_n}([0,1])$$  it follows that  $x+y\in E$.  If  $(a_n)\in C$,  by Lemma \ref{app} there
exists $b\in E$ such that $|b-x-y|<\epsilon$.
\end{proof}
\begin{Lemma}\label{IIFS}
$\overline{\bigcup\limits_{i\in \mathbb{N}^{+}}\phi_{i}(K_1+sK_2)}=K_1+sK_2$.
\end{Lemma}
\begin{proof}
 Since $$E=
\bigcup\limits_{\{\phi_{i_n}\}\in\Phi^{\infty}}\bigcap\limits_{n=1}^{\infty}\phi_{i_1}\circ\phi_{i_2}\cdots
\phi_{i_n}([0,1])$$ it follows that  $E=\bigcup\limits_{i\in \mathbb{N}^{+}}\phi_{i}(E)$, which yields that \[
\overline{E}=\overline{\bigcup\limits_{i\in \mathbb{N}^{+}}\phi_{i}(E)}=\overline{\overline{\bigcup\limits_{i\in \mathbb{N}^{+}}\phi_{i}(E)}}\supseteq \overline{\bigcup\limits_{i\in \mathbb{N}^{+}}\overline{\phi_{i}(E)} }=\overline{\bigcup\limits_{i\in \mathbb{N}^{+}}\phi_{i}(K_1+sK_2)}, \] i.e. we have
\[\overline{\bigcup\limits_{i\in \mathbb{N}^{+}}\phi_{i}(K_1+sK_2)} \subseteq K_1+sK_2.\]
Conversely, 
$E=\bigcup\limits_{i\in \mathbb{N}^{+}}\phi_{i}(E)\subseteq \bigcup\limits_{i\in \mathbb{N}^{+}}\phi_{i}(K_1+sK_2)$,  by Lemma \ref{Closure}
it follows that 
$$  K_1+sK_2\subset  \overline{\bigcup\limits_{i\in \mathbb{N}^{+}}\phi_{i}(K_1+sK_2)}.$$
\end{proof}
\begin{proof}[Proof of Theorem \ref{Main}:]
 Lemma \ref{GenerateMatchings} states that  there are at most countably many Matchings generated by $D_1$ and $D_2^{\prime}$. Suppose that  the cardinality of Matchings is infinitely  countable, then by Lemma \ref{IIFS}, $K_1+sK_2$ is an attractor of $\Phi^{\infty}$. If the cardinality is finite, then $K_1+sK_2$ is a self-similar set. The proof is  similar to Lemmas \ref{Closure} and \ref{app}. 
For this case we may not   approximate  the coding of $x+sy \in K_1+sK_2$. Indeed,  we can  directly find  a coding which is the concatenation of  infinite Matchings such that the value of this infinite coding is $x+sy$, i.e. $E=K_1+sK_2$. 
\end{proof}
Therefore, in terms of Mauldin and Urbanski's result \cite{MRD}, Lemmas \ref{IIFS} and \ref{conjugate}, we have 
\setcounter{Proposition}{12}
\begin{Proposition}\label{P=B}
For any $\theta\in[0,\pi)$, $$ \dim_{P}(Proj_{\theta}(K_1\times K_2))= \overline{\dim}_{B}(Proj_{\theta}(K_1\times K_2)).$$
\end{Proposition}
The following results were proved in \cite{SumKan}.
\setcounter{Lemma}{13}
\begin{Lemma}\label{almostequal}
If $C$ is countable,  then for any $s\in \mathbb{R}$,  $\dim_{H}(E)=\dim_{H}(K_1+sK_2)$.  
\end{Lemma}
\begin{Lemma}
Given any $k\in \mathbb{N}^{+}$.
Let $K_1$ be the attractor of the following IFS 
$$\left\{f_i(x)=\dfrac{x}{\beta^{k}}+a_i,1\leq i\leq n-1, f_n(x)=\dfrac{x}{\beta^{2k}}+a_n\right\},$$
and $K_2$ be the attractor of the following IFS 
$$\left\{g_j(x)=\dfrac{x}{\beta^{k}}+b_j,1\leq j\leq m-1, g_m(x)=\dfrac{x}{\beta^{2k}}+b_m\right\},$$
where $a_i, b_j\in \mathbb{R}, 1\leq i\leq n,1\leq j\leq m$. Then $C$ is countable. 
\end{Lemma}

\subsection{Dimension of $K_1+sK_2$}
In \cite{SumKan}, we proved the following results. 
\begin{Lemma}\label{finitecc}
If  the similarity  ratios of $K_1$ are homogeneous, denoted by $\beta^{-k}, k\in \mathbb{N}^{+}$,  and the similarity  ratios of $K_2$ have the form  $\beta^{-kp_j}, 1\leq j\leq m, p_j\in \mathbb{N}^{+},$ then $\sharp D$ is finite. 
\end{Lemma}
\begin{Lemma}\label{sss}
If $\sharp(D)$ is finite, then $K_1+sK_2$ is a self-similar set.
\end{Lemma}
\begin{proof}[\textbf{Proof of Corollary \ref{Cor}}]
 Corollary \ref{Cor} follows from Lemmas \ref{finitecc}, \ref{sss}  and Nagi and Wang's finite type condition \cite{NW}. 
\end{proof}

We are interested in the case when $K_1+sK_2$ is an attractor of some infinite iterated function system. 
For this case, we utilize Moran's idea \cite{Moran}, and find a sub-infinite iterated function system such that  the new IIFS satisfies the open set condition and the Hausdorff dimension of   two attractors  coincides. 

For convenience, we introduce  the  Vitali algorithm.
Let 
$\Phi^{\infty}=\{\phi_1,\,\phi_2,\,\phi_3,\,\phi_4,\cdots\}
$ be the IIFS generated from the set of all the Matchings. The attractor of this IIFS is 
$$E=
\bigcup\limits_{\{\phi_{i_n}\}\in\Phi^{\infty}}\bigcap\limits_{n=1}^{\infty}\phi_{i_1}\circ\phi_{i_2}\cdots\circ
\phi_{i_n}([0,1]).$$
Define
$$\Psi^{*}=\{\cup_{k=1}^{\infty}\cup_{i_1\cdots i_k}\phi_{i_1}\circ \phi_{i_2}\cdots \circ \phi_{i_k}\}.$$
Clearly, $$\Psi^{*}(E)=\{\cup_{k=1}^{\infty}\cup_{i_1\cdots i_k}\phi_{i_1}\circ \phi_{i_2}\cdots \circ \phi_{i_k}(E)\}$$ is a Vitali class of $E$ (\cite{FG}). Now we implement the Vitali process. Take any $\phi\in \Psi^{*}$, if $\phi_n$ has been selected for $1\leq n\leq k$, then we pick $\phi_{k+1}$ from $\Psi^{*}$ satisfying the following conditions,
\begin{itemize}
\item [(1)] $\phi_{k+1}(E)\cap \phi_{i}(E)=\emptyset$ for $1\leq i\leq k$. 
\item [(2)] $|\phi_{k+1}(E)|\geq 2^{-1}\sup\{|\phi(E)|:\phi\in \Psi^{*} \mbox{ and } \phi(E)\cap \phi_i(E)=\emptyset, 1\leq i\leq k\}$, where $|A|$ denotes the diameter of $A$. 
\end{itemize}
This process  is finished if  the selection of  $\phi_{k+1}$ is no longer possible.  Denote all the similitudes selected from the Vitali process by $\Psi.$
Moran \cite{Moran} proved the following theorem.
\setcounter{Theorem}{17}
\begin{Theorem}\label{Vitali}
Let $$E=
\bigcup\limits_{\{\phi_{i_n}\}\in\Phi^{\infty}}\bigcap\limits_{n=1}^{\infty}\phi_{i_1}\circ\phi_{i_2}\cdots\circ
\phi_{i_n}([0,1]),$$ and
 $$G=
\bigcup\limits_{\{\phi_{i_n}\}\in\Psi}\bigcap\limits_{n=1}^{\infty}\phi_{i_1}\circ\phi_{i_2}\cdots\circ
\phi_{i_n}([0,1]),$$
 Then 
\begin{itemize}
\item [(1)] $\mathcal{H}^s(E)=\mathcal{H}^s(G)$ for any $s$ satisfying $\sum_{\phi_i\in \Psi } r_i^s<\infty$, where $r_i$ is the similarity ratio of $\phi_i;$
\item [(2)] $\dim_{H}(E)=s$, where $s=\inf\left\{t:\sum_{\phi_i\in \Psi } r_i^t\leq 1\right\}$.
\end{itemize}
\end{Theorem}
Therefore, by means of Lemma \ref{Closure} and Theorem \ref{Vitali}, it follows that  $$\dim_{H}(Proj_{\theta}(K_1\times K_2))=\dim_{H}(K_1+sK_2)\geq \dim_{H}(E)=\dim_{H}(G), $$
which gives an lower bound of $\dim_{H}(Proj_{\theta}(K_1\times K_2))$. For the upper bound, we use the similarity dimension of $E$. 
The following lemma is standard. 
\setcounter{Lemma}{18}
\begin{Lemma}\label{similarity}
$\dim_{H}(Proj_{\theta}(K_1\times K_2))\leq s_0$, where $s_0$ is the solution of $$\sum_{\phi_i\in \Phi^{\infty}}r_i^s=1.$$
\end{Lemma}
\begin{proof}
Let $\delta>0$, there exists some $k>0$ such that 
$$|\phi_{i_1\cdots i_k}(Cov(E))|\leq \delta,$$
where   $Cov(E)$ denotes  the convex hull of $E$. 
By the definition of $E,$ it follows that $$E=\bigcup\limits_{i\in \mathbb{N}}\phi_{i}(E).$$
 Then for any $k\geq 1$, 
$$\cup_{(i_1\cdots i_k)\in \mathbb{N}^k}\phi_{i_1\cdots i_k}(Cov(E))\supset \overline{E}=K_1+sK_2.$$
Therefore, 
$$\mathcal{H}^{s_0}_{\delta}(K_1+sK_2)\leq \sum_{(i_1\cdots i_k)\in \mathbb{N}^k}|\phi_{i_1\cdots i_k}(Cov(E)) |^{s_0}=\sum_{(i_1\cdots i_k)\in \mathbb{N}^k}r_{i_1}^{s_0}\cdots r_{i_k}^{s_0}|Cov(E) |^{s_0}. $$
Note that $$\sum_{(i_1\cdots i_k)\in\mathbb{N}^k}r_{i_1}^{s_0}\cdots r_{i_k}^{s_0}|Cov(E) |^{s_0}\leq 
\left(\sum_{i=1}^{\infty}r_i^{s_0}\right)^k|Cov(E)|^{s_0}=|Cov(E)|^{s_0}<\infty.$$
\end{proof}
\begin{proof}[Proof of Theorem \ref{IIFSS}]
 Theorem \ref{IIFSS} follows from Lemma \ref{similarity}, Theorem \ref{Vitali}. 
\end{proof}
\section{One example}
In this section, we give one example to illustrates how to find the Hausdorff dimension of $Proj_{\theta}(K_1\times K_2)$ in terms of Theorems \ref{DimensionIIFS} and \ref{Vitali}.
\begin{Example}
Let $K_1=K_2$ be the attractor of the IFS $$\left\{f_1(x)=\dfrac{x}{\beta^4},\dfrac{x+\beta^8-1}{\beta^8}\right\}.$$
Suppose that $\beta>1.39$, then  for
any $\theta\in\left(\arctan \dfrac{\beta^{12}-\beta^{8}+1}{\beta^{12}-\beta^8-\beta^4}, \arctan \dfrac{\beta^8-2\beta^4}{\beta^{12}-\beta^8+1}\right)$ 
$$\dim_{H}(Proj_{\theta}(K_1\times K_2))=\dfrac{\log \sqrt{\dfrac{1+\sqrt{5}}{2}}}{\log \beta}=\dim_{H}(K_1)+\dim_{H}(K_2).$$
Let $\theta=\arctan\dfrac{\beta^{8}-1}{\beta^8-\beta^4+1}$ and $\beta>1.41$. 
Then $$\dim_{P}(Proj_{\theta}(K_1\times K_2))=\dfrac{\log \gamma}{\log \beta}<\dim_{H}(K_1)+\dim_{H}(K_2),$$ where $\gamma\approx 1.2684$ is the largest  real root of 
$$x^{20}-2x^{16}-2x^{12}+x^{8}+x^{4}-1=0.$$
\end{Example}
Denote $A=\beta^8-1,B=sA,C=A+B$.
Then $$D=\{(0000), (0000000A), (0000000B),(0000000C), (0000000B000A),(0000000A000B),\cdots\}$$
The associated  IIFS of $D$  is  $$\Phi^{\infty}=\{f(x), h_1(x), h_2(x),  \phi_{2n}(x),  \phi_{2n-1}(x),  g(x), n\geq 1\},$$
where 
$$f(x)=\dfrac{x}{\beta^4}, h_{1}(x)=\dfrac{x}{\beta^{8}}+\dfrac{A}{\beta^8},h_{2}(x)=\dfrac{x}{\beta^{8}}+\dfrac{B}{\beta^8}, g(x)=\dfrac{x}{\beta^{8}}+\dfrac{A+B}{\beta^8} $$
$$\phi_{2n-1}(x)=\dfrac{x}{\beta^{4n+8}}+\dfrac{B}{\beta^8}+\dfrac{A}{\beta^{12}}+\dfrac{B}{\beta^{16}}+\cdots+\dfrac{c(n)A+e(n)B}{\beta^{4n+8}},$$
$$\phi_{2n}(x)=\dfrac{x}{\beta^{4n+8}}+\dfrac{A}{\beta^8}+\dfrac{B}{\beta^{12}}+\dfrac{A}{\beta^{16}}+\cdots+\dfrac{c(n)B+e(n)A}{\beta^{4n+8}},n\geq 1,$$
where \begin{equation*}
c(n)=\left\lbrace\begin{array}{cc}
                  1& n \mbox{ is odd}\\

               0& n \mbox{ is even}
                \end{array}\right.
\end{equation*}
\begin{equation*}
e(n)=\left\lbrace\begin{array}{cc}
                  1& n \mbox{ is even}\\

               0& n \mbox{ is odd}
                \end{array}\right.
\end{equation*}
Let $O=(0,1+s)$, and $I=[0,1+s]$.
It is easy to check  the following statements, see Figure 1.
\begin{itemize}
\item [(1)]  $f(O)\cap h_1(O)=\emptyset$ if and only if $s<\beta^4-\beta^{-4}-1$; 
\item [(2)]  $h_1(O)\cap h_2(O)=\emptyset$ if and only if $s>\dfrac{\beta^8}{\beta^8-2}$; 
\item [(3)]  $h_2(O)\cap \phi_2(O)=\emptyset$ if and only if $s<\dfrac{\beta^{12}-2\beta^4}{\beta^{12}-\beta^8+1}$; 
\item [(4)] $\phi_{2n}(O)\cap \phi_1(O)=\emptyset$ if and only if $\dfrac{\beta^{12}-\beta^{8}+1}{\beta^{12}-\beta^8-\beta^4}<s$;
\item [(5)]  $\phi_{2n-1}(O)\cap g(O)=\emptyset$ if and only if  $ s<\beta^{8}-\beta^4-1$,  where $n\geq 1$;
\item [(6)]  $\phi_{2n}(O)\cap \phi_{2n+2}(O)=\emptyset$ and $\phi_{2n-1}(O)\cap \phi_{2n+1}(O)=\emptyset$ if and only if $$\dfrac{\beta^4}{\beta^8-\beta^4-1}<s<\beta^4-\beta^{-4}-1,$$ where $n\geq 1$.
\end{itemize}
\begin{figure}[h]\label{figure1}
\centering
\begin{tikzpicture}[scale=5]
\draw(0,0)node[below]{\scriptsize $0$}--(3,0)node[below]{\scriptsize$1+s$};

\draw(0,-0.3)node[below]{\scriptsize $0$}--(0.3,-0.3);
\node [label={[xshift=0.8cm, yshift=-1.5cm]$f(I)$}] {};

\draw(0.4,-0.3)--(0.55,-0.3);
\node [label={[xshift=2.5cm, yshift=-1.5cm]$h_1(I)$}] {};

\draw(0.6,-0.3)--(0.75,-0.3);
\node [label={[xshift=3.5cm, yshift=-1.5cm]$h_2(I)$}] {};

\draw(2.85,-0.3)--(3,-0.3)node[below]{\scriptsize$1+s$};
\node [label={[xshift=14.5cm, yshift=-1.5cm]$g(I)$}] {};

\draw(0.8,-0.3)--(0.9,-0.3);
\node [label={[xshift=4.5cm, yshift=-1.5cm]$\phi_2(I)$}] {};

\draw(1,-0.3)--(1.05,-0.3);
\node [label={[xshift=5.5cm, yshift=-1.5cm]$\phi_4(I)$}] {};

\draw(1.8,-0.3)--(1.9,-0.3);
\node [label={[xshift=9.5cm, yshift=-1.5cm]$\phi_1(I)$}] {};

\draw(2.0,-0.3)--(2.05,-0.3);
\node [label={[xshift=10.5cm, yshift=-1.5cm]$\phi_3(I)$}] {};

\node [label={[xshift=6.5cm, yshift=-1.9cm]$\cdots\cdots\cdots$}] {};

\node [label={[xshift=12.5cm, yshift=-1.9cm]$\cdots\cdots\cdots$}] {};
\end{tikzpicture}\caption{First iteration}
\end{figure}
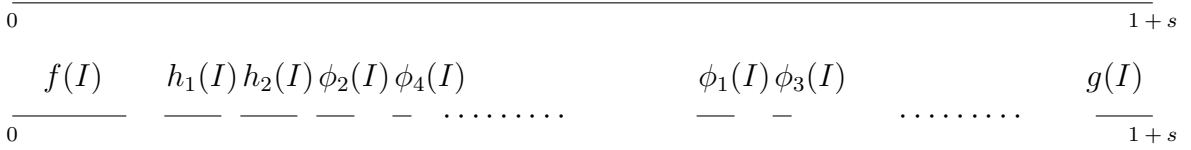

Hence, if $\beta>1.39$ then the following inequalities hold $$\dfrac{\beta^4}{\beta^8-\beta^4-1}<\dfrac{\beta^8}{\beta^8-2}<\dfrac{\beta^{12}-\beta^{8}+1}{\beta^{12}-\beta^8-\beta^4}<\dfrac{\beta^{12}-2\beta^4}{\beta^{12}-\beta^8+1}<\beta^4-\beta^{-4}-1<\beta^8-\beta^{4}-1.$$ In other words,   let $\theta\in\left(\arctan \dfrac{\beta^{12}-\beta^{8}+1}{\beta^{12}-\beta^8-\beta^4}, \arctan \dfrac{\beta^8-2\beta^4}{\beta^{12}-\beta^8+1}\right)$, then  
 $\Phi^{\infty}$  satisfies the open set condition with  the open set $(0,1+s)$. 
 In terms of Theorem \ref{DimensionIIFS} and Lemma \ref{almostequal}, it follows that $$\dim_{H}(Proj_{\theta}(K_1\times K_2))=\dfrac{\log \gamma^{*}}{\log \beta},$$
where $\gamma^{*}$ is largest  real root of $x^{12}-2x^8-2x^4+1=0.$
It is easy to check that $$\gamma^{*}=\sqrt{\dfrac{1+\sqrt{5}}{2}}.$$ 
For the second case, 
note that  $s=\tan\theta=\dfrac{\beta^{8}-1}{\beta^8-\beta^4+1}$ if and only if $h_2\circ g=\phi_2\circ f.$ Moreover, if  $\beta>1.41$, 
then $$\dfrac{\beta^8}{\beta^8-2}<\dfrac{\beta^{8}-1}{\beta^8-\beta^4+1}<\dfrac{\beta^{12}-\beta^{8}+1}{\beta^{12}-\beta^8-\beta^4}<\dfrac{\beta^{12}-2\beta^4}{\beta^{12}-\beta^8+1}.$$ In this case the IIFS does not satisfy the open set condition, see the first iteration in Figure 2.  
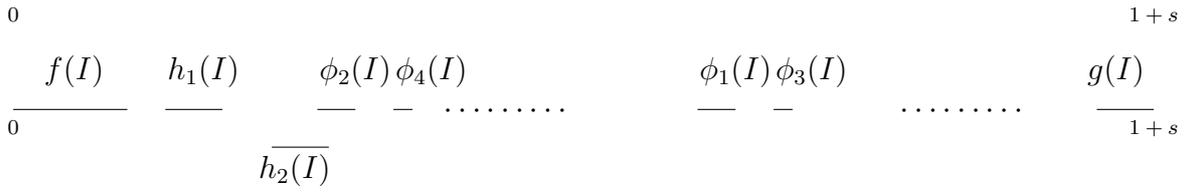
\begin{figure}[h]\label{figure1}
\centering
\begin{tikzpicture}[scale=5]
\draw(0,0)node[below]{\scriptsize $0$}--(3,0)node[below]{\scriptsize$1+s$};

\draw(0,-0.3)node[below]{\scriptsize $0$}--(0.3,-0.3);
\node [label={[xshift=0.8cm, yshift=-1.5cm]$f(I)$}] {};

\draw(0.4,-0.3)--(0.55,-0.3);
\node [label={[xshift=2.5cm, yshift=-1.5cm]$h_1(I)$}] {};

\draw(0.68,-0.4)--(0.83,-0.4);
\node [label={[xshift=3.7cm, yshift=-2.8cm]$h_2(I)$}] {};

\draw(2.85,-0.3)--(3,-0.3)node[below]{\scriptsize$1+s$};
\node [label={[xshift=14.5cm, yshift=-1.5cm]$g(I)$}] {};

\draw(0.8,-0.3)--(0.9,-0.3);
\node [label={[xshift=4.5cm, yshift=-1.5cm]$\phi_2(I)$}] {};

\draw(1,-0.3)--(1.05,-0.3);
\node [label={[xshift=5.5cm, yshift=-1.5cm]$\phi_4(I)$}] {};

\draw(1.8,-0.3)--(1.9,-0.3);
\node [label={[xshift=9.5cm, yshift=-1.5cm]$\phi_1(I)$}] {};

\draw(2.0,-0.3)--(2.05,-0.3);
\node [label={[xshift=10.5cm, yshift=-1.5cm]$\phi_3(I)$}] {};

\node [label={[xshift=6.5cm, yshift=-1.9cm]$\cdots\cdots\cdots$}] {};

\node [label={[xshift=12.5cm, yshift=-1.9cm]$\cdots\cdots\cdots$}] {};
\end{tikzpicture}\caption{First iteration}
\end{figure}
 We make use of the Vitali process to find the  $\Psi.$ 
It is not difficult to check that in $\Phi^{\infty}$
only  for the pair $(h_2, \phi_2)$, $h_2(O)\cap \phi_2(O)\neq \emptyset$. For other similitudes $$(S_1(x), S_2(x))\neq (h_2, \phi_2) ,$$ ($S_i(x)\in \Phi^{\infty}, i=1, 2$), $S_1(O)\cap S_2(O)=\emptyset$, see the first iteration in Figure 2. 

Hence, we implement the Vitali process and find all the similitudes of $\Psi$, i.e. 
$$\Psi=\{\Phi^{\infty}\setminus\{\phi_2\}\}\cup \cup_{k=1}^{\infty}\{\phi_{2^{k}}(\Phi^{\infty}\setminus\{\phi_2,  f\})\},$$
where $\phi_{2^{k}}(\Phi^{\infty}\setminus\{\phi_2,  f\})=\{\phi_{2^k}\circ h:h\in \Phi^{\infty}\setminus\{\phi_2,  f\})\}$ for any $k\geq 1.$

By Theorem \ref{Vitali} and Lemma \ref{almostequal}, it follows that $\dim_{H}(Proj_{\theta}(K_1\times K_2))=\dfrac{\log \gamma}{\log \beta},$ where $\gamma\approx 1.2684$ is the largest  real root of 
$$x^{20}-2x^{16}-2x^{12}+x^{8}+x^{4}-1=0.$$
\section{Final remarks}
We can obtain  the following stronger  result.
\begin{Theorem}
Take any $K_1, K_2,\cdots, K_n\in \mathcal{A}$ and any real numbers $p_1,\cdots, p_n$. If there are some $1\leq i\neq j\leq n$ such that   $p_{i}, p_j\neq 0$, then 
$$p_1K_1+p_2K_2+\cdots +p_nK_n=\left\{\sum_{i=1}^{n}p_ix_i:x_i\in K_i, 1\leq i\leq n\right\}$$ is 
a self-similar set or an attractor of some infinite iterated function system. 
\end{Theorem}
The proof of this result is similar to Theorem \ref{Main}. Therefore, we can consider the set
$$Proj_\theta(K_1\times K_2\times\cdots \times K_n),$$  and obtain similar result as Theorem \ref{Main}. 
Finally, we pose the following question:
\setcounter{Question}{1}
\begin{Question}
Take $K_1, K_2\in \mathcal{A}$ and $\theta\in[0,\pi)$. If $$\dim_{H}(Proj_{\theta}(K_1\times K_2))=\dim_{H}(K_1)+\dim_{H}(K_2), $$
 then must the IFS (IIFS) of the attractor, which is similar to $Proj_{\theta}(K_1\times K_2)$,   satisfy the open  set condition? 
\end{Question}

\section*{Acknowledgements}
The work is supported by National Natural Science Foundation of China (Nos.11701302,

11671147). The work is also supported by K.C. Wong Magna Fund in Ningbo University.

%\bibliographystyle{plain}
%\bibliography{OnUnivoquePointsForSelfSimilarSets}

\end{document}